\documentclass[11pt]{amsart}
\usepackage{hyperref,enumitem}
\usepackage{amssymb,amsmath,amsrefs}

\makeatletter
\let\@@pmod\pmod
\DeclareRobustCommand{\pmod}{\@ifstar\@pmods\@@pmod}
\def\@pmods#1{\mkern4mu({\operator@font mod}\mkern 6mu#1)}
\makeatother

\newtheorem{thm}{Theorem}
\newtheorem{conj}[thm]{Conjecture}
\newtheorem{lem}[thm]{Lemma}
\newtheorem{prop}[thm]{Proposition}
\newtheorem{cor}[thm]{Corollary}
\theoremstyle{definition}
\newtheorem*{rem}{Remark}
\newtheorem*{rems}{Remarks}
\newtheorem{ex}[thm]{Example}

\newcommand{\R}{\mathbb{R}}

\newcommand{\Q}{\mathbb{Q}}
\newcommand{\Z}{\mathbb{Z}}
\newcommand{\F}{\mathbb{F}}

\renewcommand{\to}{\rightarrow}
\renewcommand{\epsilon}{\varepsilon}

\newcommand{\disc}{\operatorname{disc}}
\newcommand{\lc}{\operatorname{lc}}
\newcommand{\lcm}{\operatorname{lcm}}
\newcommand{\ord}{\operatorname{ord}}
\newcommand{\sqf}{\operatorname{sqf}}

\begin{document}

\title[Twists of hyperelliptic curves]{Twists of hyperelliptic curves by integers in progressions modulo $p$}

\author[D. Krumm]{David Krumm}
\address{Mathematics Department\\
Reed College}
\email{dkrumm@reed.edu}
\urladdr{http://maths.dk}

\author[P. Pollack]{Paul Pollack}
\address{Department of Mathematics\\
University of Georgia}
\email{pollack@uga.edu}
\urladdr{http://pollack.uga.edu}

\thanks{The second author was supported by NSF award DMS-1402268.}

\maketitle

\begin{abstract}
Let $f(x)$ be a nonconstant polynomial with integer coefficients and nonzero discriminant. We study the distribution modulo primes of the set of squarefree integers $d$ such that the curve $dy^2=f(x)$ has a nontrivial rational or integral point.
\end{abstract}

\section{Introduction}

Let $f(x)\in\Z[x]$ be a nonconstant polynomial with nonzero discriminant, and let $C$ be the hyperelliptic curve over $\Q$ defined by $y^2=f(x)$. For every squarefree integer $d$, let $C_d$ denote the quadratic twist $dy^2=f(x)$. The main object of interest in this article is the set $S_{\Q}(f)$ consisting of all squarefree integers $d$ such that $C_d$ has a nontrivial rational point, i.e., an affine rational point $(x_0,y_0)$ with $y_0\ne 0$. Specifically, we are interested in the following conjecture, which was proposed by the first author in \cite{krumm}.

\begin{conj}\label{krumm_conjecture}
For every large enough prime $p$, and every integer $r$ not divisible by $p$, there exist infinitely many $d\in S_{\Q}(f)$ such that $d\equiv r\pmod p$.
\end{conj}

This conjecture is proved in \cite{krumm} in the case where $\deg f\le 2$. Furthermore, when $\deg f=3$, or when $\deg f=4$ and $f(x)$ has a rational root, the conjecture is shown to follow from the Parity Conjecture for elliptic curves over $\Q$. In this paper we explain how to leverage known results on squarefree values of polynomials and binary forms to prove the following two theorems.

First, using work of Granville \cite{granville} we show that Conjecture \ref{krumm_conjecture} follows from the abc conjecture; in fact, the latter can be used to prove a stronger statement. Let us denote by $S_{\Z}(f)$ the set of all squarefree integers $d$ such that $C_d$ has a nontrivial \emph{integral} point.

\begin{thm}\label{abc_thm}
The abc conjecture implies that for every large enough prime $p$, and every integer $r$ not divisible by $p$, there exist infinitely many $d\in S_{\Z}(f)$ such that $d\equiv r\pmod p$.
\end{thm}

Second, we prove an unconditional result by using work of Greaves \cite{greaves}.

\begin{thm}\label{unconditional_thm}
Conjecture \ref{krumm_conjecture} holds if every irreducible factor of $f(x)$ over $\Q$ has degree at most six.
\end{thm}

In addition, we consider the distribution of elements of $S_{\Z}(f)$ modulo $p$ when $p$ is a ``small" prime, by which we mean that at least one of the conditions in \eqref{p_sufficient_conditions} is not satisfied.

\section{Assuming abc: Proof of Theorem \ref{abc_thm}}

We will need the following special case of Theorem 1 in \cite{granville}.

\begin{prop}[Granville]\label{granville_prop} Assume the abc conjecture is true. Let $g(x)$ be a nonconstant polynomial with integer coefficients and nonzero discriminant, and suppose that there is no prime $p$ such that $p|g(n)$ for all integers $n$. Then there exist infinitely many integers $n$ such that $g(n)$ is squarefree.
\end{prop}

Recall that an integer $k$ is called a \emph{fixed divisor} of $f(x)$ if $k|f(n)$ for every integer $n$. The set of all fixed divisors of $f(x)$ is finite, and therefore has a largest element, which we denote by $D$. It is a simple exercise to show that $D$ is maximal also in the sense that every fixed divisor of $f(x)$ divides $D$.

Let $p$ be a prime number, let $\ord_p$ denote the $p$-adic valuation on $\Z$, and let $\epsilon\in\{0,1\}$ be the parity of $\ord_p(D)$. For every integer $r\not\equiv 0\pmod p$ and for every integer $v\ge 0$ we define a statement $S(r,v)$ as follows:
\vspace{-1mm}
\begin{equation}\label{S_statement}
S(r,v)\;
\begin{cases}
\text{There exist $h,x_0,y_0\in\Z$ satisfying}\\
\;\;\bullet\;\;hy_0^2\equiv f(x_0)\pmod*{p^{2(v+\epsilon) + 1}},\\
\;\;\bullet\;\;\ord_p(y_0)=v+\epsilon,\text{ and}\\
\;\;\bullet\;\;h\equiv r\pmod* p.
\end{cases}
\end{equation}

The proof of the following proposition establishes the key ideas to be used throughout this article.

\begin{prop}[Assuming abc]\label{S_reverse_direction} Let $r$ be an integer not divisible by $p$.
Suppose that $S(r,v)$ holds true for some $v\ge 0$, and that $f(x)$ has an irreducible factor whose discriminant is not divisible by $p$. Then there exist infinitely many integers $d\in S_{\Z}(f)$ such that $d\equiv r\pmod p$.
\end{prop}
\begin{proof}
For every nonzero rational number $x$ we denote by $\sqf(x)$ the squarefree part of $x$, i.e., the unique squarefree integer representing the coset of $x$ in $(\Q^{\ast})/(\Q^{\ast})^2$. By definition of $\epsilon$, we have $\ord_p(D)=2k+\epsilon$ for some nonnegative integer $k$. Writing $D=\sqf(D)t^2$, it is necessarily the case that $\ord_p(t)=k$ and $\ord_p(\sqf(D))=\epsilon$; thus, we may write $t=p^ku$, where $p\nmid u$, and $\sqf(D)=p^{\epsilon}\delta$ for some squarefree integer $\delta$ not divisible by $p$.

Since $S(r,v)$ holds true, there exist integers $h, x_0,$ and $y_0$ satisfying the properties listed in \eqref{S_statement}. In particular, $\ord_p(y_0)=v+\epsilon$, so we may write $y_0=p^{v+\epsilon}z_0$, where $p\nmid z_0$.

By the Chebotarev density theorem\footnote{See Lemma 4.4 in \cite{krumm} for details. The crucial fact we use here is that if $h(x)$ is an irreducible factor of $f(x)$ such that $p\nmid\disc h(x)$, then the intersection of the splitting field of $h(x)$ and the cyclotomic field $\Q(\zeta_p)$ is $\Q$.}, there exists a prime $q\nmid D$ such that $qu\equiv z_0\pmod p$ and $f(x)$ has a simple root modulo $q$. The latter property ensures, via Hensel's lemma\footnote{Let $\alpha$ be a simple root of $f(x)$ modulo $q$. Hensel lifting allows us to find an integer $\beta$ such that $\beta\equiv\alpha\pmod q$ and $f(\beta)\equiv 0\pmod{q^3}$. Then $m=\beta+q^2$ satisfies $q^2\|f(m)$.}, that there exists $m\in\Z$ such that $q^2\|f(m)$.

For every prime $s\ne p$ dividing $D$, let $e_s=\ord_s(D)$ and let $n_s$ be an integer such that $f(n_s)\not\equiv 0\pmod{s^{e_s+1}}$. (Such an integer must exist, for otherwise $\lcm(s^{e_s+1},D)=sD$ would be a fixed divisor of $f(x)$, contradicting the maximality of $D$.)

Choose $b\in\Z$ satisfying
\begin{align*}
b&\equiv x_0\pmod*{p^{2(v+\epsilon)+1}},\\
b&\equiv m\pmod*{q^3},\text{ and}\\
b&\equiv n_s\pmod*{s^{e_s+1}}\text{ for every prime } s|D, s\ne p.
\end{align*}
Let $a=q^2p^{2(v+\epsilon)+1}\prod_{s}s^{e_s+1}$, and define a polynomial $g(x)$ by the equation
\[\Delta\cdot g(x)=f(ax+b),\text{ where } \Delta=Dq^2p^{2(v-k)+\epsilon}.\]
Note that $v\ge k$, so that $\Delta\in\Z$. Indeed, the properties in \eqref{S_statement} imply that $\ord_p(f(x_0))=2(v+\epsilon)$. Since $D|f(x_0)$, then $\ord_p(D)\le\ord_p(f(x_0))$, so $2k+\epsilon\le 2v+2\epsilon$, and therefore $k\le v$.

We claim that $g(x)$ satisfies all the hypotheses of Proposition \ref{granville_prop}. A Taylor expansion shows that $f(ax+b)=f(b)+a\cdot P(x)$ for some polynomial $P(x)\in\Z[x]$. Thus, in order to show that $g(x)\in\Z[x]$ it suffices to show that $\Delta$ divides both $f(b)$ and $a$. From the definitions it follows easily that $\ord_{\ell}(a)\ge\ord_{\ell}(\Delta)$ for every prime $\ell$ dividing $\Delta$, so $\Delta|a$. Similarly, the definition of $b$ implies that $\Delta|f(b)$. Hence $g(x)\in\Z[x]$. Now suppose that $\ell$ is a fixed prime divisor of $g(x)$. We claim that $\ell\nmid a$. If $\ell=q$, then $q|g(q)$, so $q^3|f(aq+b)$. However, $f(aq+b)\equiv f(b)\equiv f(m)\not\equiv 0\pmod{q^3}$. Thus $\ell\ne q$. Suppose now that $\ell$ is one of the primes $s$, and let $n\in\Z$. Then $s|g(n)$, so $s^{e_s+1}|f(an+b)$. However, $f(an+b)\equiv f(b)\equiv f(n_s)\not\equiv 0\pmod{s^{e_s+1}}$. Thus $\ell\ne s$. Similarly, we can show that $p$ does not divide $g(n)$ for any integer $n$. For if $p|g(n)$, then $f(an+b)\equiv 0\pmod{p^{2(v+\epsilon)+1}}$. However, $f(an+b)\equiv f(b)\equiv f(x_0)\not\equiv 0\pmod*{p^{2(v+\epsilon)+1}}$.
This proves that $\ell\nmid a$. Now, since the map $x\mapsto (ax+b)$ is invertible modulo $\ell$, the assumption that $\ell$ is a fixed divisor $g(x)$ implies that it is also a fixed divisor of $f(x)$. It follows that $\ell| D$, but this has already been ruled out above. Therefore, $g(x)$ has no fixed prime divisor. Finally, $\disc g(x)\ne 0$ since $\disc f(x)\ne 0$ by assumption.

As shown above, neither $p$ nor any of the primes $s$ can divide $g(n)$ for any integer $n$. Thus,\begin{equation}\label{coprime_values}
\gcd(g(n), pD)=1\text{ for every integer }n.
\end{equation}

The last step in the proof is to show that there is a well-defined map
\[\psi:\{n\in\Z\;\vert\;g(n)\text{ is squarefree}\}\to\{d\in S_{\Z}(f)\;\vert\; d\equiv r\pmod* p\}\]
given by $n\mapsto \delta g(n)$. Note that the domain of $\psi$ is infinite by Proposition \ref{granville_prop}. Let $n\in\Z$ be such that $g(n)$ is squarefree. Tracking through the definitions we find that
\begin{equation}\label{f_value}
f(ax+b)=\delta g(x)(qu)^2p^{2v+2\epsilon}.
\end{equation}

By \eqref{coprime_values} we have $\gcd(g(n), \delta)=1$, so \eqref{f_value} implies that \[d:=\sqf(f(an+b))=\delta g(n).\]
Reducing \eqref{f_value} modulo $p^{2(v+\epsilon)+1}$  and recalling that $y_0=p^{v+\epsilon}z_0$, we obtain
\[d(qu)^2p^{2v+2\epsilon}\equiv f(b)\equiv f(x_0)\equiv hy_0^2\equiv hp^{2v+2\epsilon}z_0^2\pmod*{p^{2(v+\epsilon)+1}}.\]
It follows that $d(qu)^2\equiv hz_0^2\pmod* p$. Since $qu\equiv z_0\pmod p$ by construction and $h\equiv r\pmod p$ by the assumptions in \eqref{S_statement}, this implies that $d\equiv r\pmod p$. Moreover, it is clear from the definitions that $d\in S_{\Z}(f)$. Thus, we have shown that the map $\psi$ is well defined.

The equation $g(x)=g(y)$ has only finitely many integral solutions with $x\ne y$, since the polynomial map $x\mapsto g(x)$ is injective outside some bounded interval in the real line. Hence, the fact that the domain of $\psi$ is infinite implies that its image is infinite as well. This completes the proof.
\end{proof}

\begin{rems}\mbox{}
\begin{enumerate}[leftmargin=7mm]
\item[(i)] Proposition \ref{granville_prop} is known to hold unconditionally if every irreducible factor of $f(x)$ has degree at most three (see \cite[Chap. 4]{hooley}). Our arguments show that Proposition \ref{S_reverse_direction} also holds unconditionally in this case.
\item[(ii)] The version of Proposition \ref{granville_prop} given in \cite{granville} states that the number of positive integers $n\le B$ such that $g(n)$ is squarefree is asymptotic to $\kappa B$ (as $B\to\infty$) for some positive constant $\kappa$. Modifying the proof of Proposition \ref{S_reverse_direction} appropriately to take advantage of this, one can show that
\[\#\{d\in S_{\Z}(f): |d|\le B \text{ and }d\equiv r\pmod* p\}\gg B^{1/\deg f}.\]
\end{enumerate}
\end{rems}

\begin{cor}[Assuming abc]\label{large_p_cor}
Let $r$ be an integer not divisible by $p$. Suppose that $p\nmid D$, $p\nmid\disc f(x)$, and $ry_0^2\equiv f(x_0)\pmod p$ for some integers $x_0,y_0$ with $p\nmid y_0$. Then there exist infinitely many integers $d\in S_{\Z}(f)$ such that $d\equiv r\pmod p$.
\end{cor}
\begin{proof}
The hypotheses imply that the statement $S(r,0)$ holds true. The result then follows immediately from Proposition \ref{S_reverse_direction}.
\end{proof}

\subsection*{Proof of Theorem \ref{abc_thm}}
Assuming the abc conjecture, we must show that for every large enough prime $p$, and every integer $r$ not divisible by $p$, there exist infinitely many $d\in S_{\Z}(f)$ such that $d\equiv r\pmod p$. Let $\lc(f)$ be the leading coefficient of $f(x)$, and let $g$ be the genus of the curve $y^2=f(x)$. Suppose that $p$ is a prime satisfying
\begin{equation}\label{p_sufficient_conditions}
p\nmid\lc(f),\;p\nmid D,\;p\nmid\disc f(x),\;\text{and }p>4g^2+6g+4.
\end{equation}
Let $r$ be an integer not divisible by $p$. The Hasse-Weil bound implies that every smooth projective curve of genus $g$ over $\F_{p}$ has at least $2g+5$ points defined over $\F_p$; in particular, this applies to the hyperelliptic curve over $\F_p$ defined by $ry^2=f(x)$. This curve can have at most $2g+4$ trivial points defined over $\F_p$, so it must have a nontrivial point. Applying Corollary \ref{large_p_cor} we obtain the desired result.
\qed

\section{The case of small primes $p$}

Let $R(p)\subseteq\F_p^{\ast}$ be the set consisting of all the nonzero residue classes modulo $p$ which are represented in the set $S_{\Z}(f)$. We have shown that if $p$ is large enough, then $R(p)=\F_p^{\ast}$. In this section we discuss the problem of determining $R(p)$ when $p$ is a ``small" prime, meaning that the conditions \eqref{p_sufficient_conditions} are not all satisfied.

\begin{lem}\label{square_classes_lem}
Let $r$ be an integer not divisible by $p$, and let $v$ be a nonnegative integer. Suppose that $S(r,v)$ holds. Then $S(a,v)$ holds for every integer $a$ in the same square class as $r$ modulo $p$.
\end{lem}
\begin{proof}
Let $h, x_0,$ and $y_0$ be integers satisfying the conditions in \eqref{S_statement}. Let $g$ be a primitive root modulo $p$, and let $z$ be a multiplicative inverse of $g$ modulo $p^{2(v+\epsilon)+1}$. By hypothesis, $a\equiv rg^{2k}\pmod p$ for some positive integer $k$. From the definitions it follows that

\begin{itemize}[itemsep=2mm]
\item $hg^{2k}(z^ky_0)^2\equiv hy_0^2\equiv f(x_0)\pmod{p^{2(v+\epsilon)+1}}$,
\item $\ord_p(z^ky_0)=\ord_p(y_0)=v+\epsilon$, and
\item $hg^{2k}\equiv rg^{2k}\equiv a\pmod p$.
\end{itemize}
Hence, $S(a,v)$ holds.
\end{proof}

\begin{prop}[Assuming abc]\label{reduction_possibilities}
Suppose that $f(x)$ has an irreducible factor whose discriminant is not divisible by $p$. Then $R(p)$ is either empty or equal to one of the sets $\F_p^{\ast}$, $(\F_p^{\ast})^2$, or $\F_p^{\ast}\setminus(\F_p^{\ast})^2$.
\end{prop}
\begin{proof}
We claim that if $R(p)$ contains a square, then $R(p)\supseteq(\F_p^{\ast})^2$. Let $a$ and $r$ be nonzero quadratic residues modulo $p$, and suppose that there exists $d\in S_{\Z}(f)$ such that $d\equiv r\pmod p$. Then we have $dy_0^2=f(x_0)$ for some integers $x_0,y_0$ with $y_0\ne 0$. Letting $v=\ord_p(y_0)-\epsilon$, it is easy to verify that $v\ge 0$ and $S(r,v)$ holds. By Lemma \ref{square_classes_lem}, $S(a,v)$ also holds. Hence, by Proposition \ref{S_reverse_direction}, there exists $d'\in S_{\Z}(f)$ such that $d'\equiv a\pmod p$. This proves the claim. A similar argument shows that if $R(p)$ contains a nonsquare, then $R(p)\supseteq\F_p^{\ast}\setminus(\F_p^{\ast})^2$.

Suppose that $R(p)$ is nonempty. If $R(p)$ contains only squares, then the above argument implies that $R(p)=(\F_p^{\ast})^2$; similarly, if $R(p)$ contains only nonsquares, then $R(p)=\F_p^{\ast}\setminus(\F_p^{\ast})^2$. Finally, if $R(p)$ contains both a square and a nonsquare, then $R(p)=\F_p^{\ast}$.
\end{proof}

We now provide examples in which the various possibilities of Proposition \ref{reduction_possibilities} occur with small primes $p$.

\begin{ex}
Let $p$ be any prime such that $p\equiv 3\pmod 4$, and consider the polynomial $f(x)=(x^2+1)((x^p-x)^2+p)$. Note that $f(x)$ has a repeated root modulo $p$, so that $p|\disc f(x)$, and $p$ is a small prime for $f(x)$. We have $\ord_p(f(n))=1$ for every integer $n$, which implies that $p|\sqf(f(n))$ for all $n$. Hence, every element of $S_{\Z}(f)$ is divisible by $p$, and $R(p)=\emptyset$.
\end{ex}

\begin{ex}
Let $p$ be an arbitrary prime, and consider the polynomial $f(x)=x^p-x+1$. Note that $p$ is small for $f(x)$. We claim that $R(p)=(\F_p^{\ast})^2$. Let $r$ be an integer not divisible by $p$, and suppose that $d\in S_{\Z}(f)$ satisfies $d\equiv r\pmod p$. Then $dy_0^2=x_0^p-x_0+1$ for some integers $x_0,y_0$. Reducing modulo $p$ we obtain $ry_0^2\equiv 1\pmod p$, from which it follows that $r$ is a square modulo $p$. Thus, $R(p)\subseteq(\F_p^{\ast})^2$. Conversely, if $r$ is a nonzero square modulo $p$, then $ry_0^2\equiv 1\equiv f(x_0)\pmod p$ for some integer $y_0$ and for every integer $x_0$. Since $p\nmid D=1$ and $p\nmid\disc f(x)$, Corollary \ref{large_p_cor} implies that there exists $d\in S_{\Z}(f)$ such that $d\equiv r\pmod p$. Hence, $R(p)=(\F_p^{\ast})^2$, as claimed. A similar argument shows that if we define $f(x)=x^p-x+a$, where $a$ is a quadratic nonresidue modulo $p$, then $R(p)=\F_p^{\ast}\setminus(\F_p^{\ast})^2$.
\end{ex}

\begin{ex}
Let $p$ be prime, let $v$ be a nonnegative integer, and consider
\[f(x)=x(x^p-x)^{2v+2}+p^{2v+1}x.\]
We will show that $R(p)=\F_p^{\ast}$. Note that $p|\disc f(x)$, so $p$ is small for $f(x)$. Clearly, $p^{2v+1}$ is a fixed divisor of $f(x)$, so $p^{2v+1}|D$; in fact $p^{2v+1}\|D$ since $p^{2v+2}\nmid f(1)$. In particular, the parity of $\ord_p(D)$ is $\epsilon=1$. The statement $S(r,v)$ can now be seen to hold for every integer $r\not\equiv 0\pmod p$: indeed,
\[r(p^{v+1})^2\equiv f(rp)\pmod*{p^{2v+3}}.\]
Moreover, $f(x)$ has an irreducible factor (namely, $x$) whose discriminant is not divisible by $p$. Thus, by Proposition \ref{S_reverse_direction}, there exists $d\in S_{\Z}(f)$ such that $d\equiv r\pmod p$. We conclude that $R(p)=\F_p^{\ast}$.
\end{ex}

In the last example we show that when the discriminant condition in Proposition \ref{reduction_possibilities} is not satisfied, the conclusion may not hold.
\begin{ex}
Let $p$ be an odd prime, and let $f(x)$ be the $p$-th cyclotomic polynomial. Then $f(x)$ is irreducible and $p|\disc f(x)$. We will show that $R(p)=\{1\}$. Clearly $1\in R(p)$ because the curve $y^2=f(x)$ has a nontrivial integral point, namely $(0,1)$. Now suppose that $d\in S_f(\Z)$ is not divisible by $p$. We have $d>0$ because $f(x)$ only takes positive values for $x\in\R$. If $q$ is any prime dividing $d$, then $f(x)$ has a simple root modulo $q$. Let $K$ denote the cyclotomic field $\Q[x]/(f(x))$. By the Dedekind-Kummer theorem in algebraic number theory (Proposition 8.3 in \cite{neukirch}), some prime (and therefore every prime) of $\mathcal O_K$ lying over $(q)$ has ramification index and residue degree equal to 1. Hence, $q$ splits completely in $K$. It follows that $q\equiv 1\pmod p$ (see Corollary 10.4 in \cite{neukirch}, for instance). Since $d>0$ and every prime divisor of $d$ is congruent to 1 modulo $p$, then $d\equiv 1\pmod p$. Therefore, $R(p)=\{1\}$.
\end{ex}

\section{An unconditional result: proof of Theorem \ref{unconditional_thm}}

We will need the following special case of the main theorem in \cite{greaves}.
\begin{prop}[Greaves]\label{greaves_prop}
Let $F(x,y)\in\Z[x,y]$ be a binary form of degree $d$ with nonzero discriminant, and suppose that the coefficient of $y^d$ in $F(x,y)$ is nonzero. Let $A,B,M$ be integers with $M>0$. Assume that for every prime $\ell$ there exist integers $\alpha$ and $\beta$ such that
\begin{equation}\label{greaves_congruences}
\alpha\equiv A\pmod* M,\; \beta\equiv B\pmod* M,\text{ and }\;\ell^2\nmid F(\alpha,\beta).
\end{equation}
If every irreducible factor of $F(x,y)$ has degree at most six, then there exist infinitely many pairs of integers $\alpha,\beta$ such that $\alpha\equiv A\pmod M$, $\beta\equiv B\pmod M$, and $F(\alpha,\beta)$ is squarefree.
\end{prop}

\begin{rem}
The result in \cite{greaves} assumes that $F(x,y)$ has nonzero terms in both $x^d$ and $y^d$. To obtain Proposition \ref{greaves_prop}, one should apply the result of \cite{greaves} with $F(x,y)$ replaced with $F(x,kx+y)$ for an integer $k$ chosen so that the coefficient of $x^d$ is nonzero.
\end{rem}

\begin{prop}\label{unconditional_prop}
Let $r$ be an integer not divisible by $p$.
Suppose that $S(r,v)$ holds true for some $v\ge 0$, and that $f(x)$ has an irreducible factor whose discriminant is not divisible by $p$. Moreover, suppose that $\deg f\ge 3$ and that every irreducible factor of $f(x)$ has degree at most six. Then there exist infinitely many integers $d\in S_{\Q}(f)$ such that $d\equiv r\pmod p$.
\end{prop}
\begin{proof}
The hypotheses allow us to define a polynomial $g(x)$ as in the proof of Proposition \ref{S_reverse_direction}; we will use here the notation introduced in that proof. Let $G(x,y)$ be the homogenization of $g(x)$, $\partial=\deg g$, and $F(x,y)=y^{\sigma}G(x,y)$, where $\sigma\in\{0,1\}$ is the parity of $\partial$. We have $\disc F\ne 0$ since $\disc g(x)\ne 0$. Note that $g(0)=f(b)/\Delta$, and $f(b)\ne 0$ because $f(b)\equiv f(m)\not\equiv 0\pmod{q^3}$. It follows that the coefficient of $y^{\partial+\sigma}$ in $F(x,y)$ is nonzero.

We will apply Proposition \ref{greaves_prop} with $A=q, B=1, M=pD$. We must show that for every prime $\ell$ there exist $\alpha, \beta\in\Z$ satisfying \eqref{greaves_congruences}. By \eqref{coprime_values} we have $\gcd(g(q),pD)=1$. Thus, if $\ell|pD$, then $\ell\nmid g(q)=F(q,1)$, so we may take $\alpha=q,\beta=1$. For $\ell=q$, we have $q\nmid g(q)=F(q,1)$, as shown in the proof of Proposition \ref{S_reverse_direction}. Suppose now that $\ell\nmid pqD$, so that $\ell\nmid a$. We claim that there exists $\alpha\equiv q\pmod {pD}$ such that $\ell\nmid F(\alpha,1)$. If not, then $\ell|f(a\alpha+b)$ for every such $\alpha$. Since $a$ is invertible modulo $\ell$, this implies that $\ell$ is a fixed divisor of $f(x)$, and hence divides $D$, which is a contradiction.

Let $P$ be the set of all pairs of integers $(\alpha, \beta)$ such that $\alpha\equiv q\pmod{pD}$, $\beta\equiv 1\pmod{pD}$, and $F(\alpha,\beta)$ is squarefree. By Proposition \ref{greaves_prop}, $P$ is an infinite set. We claim that there is a well-defined map 
\[\psi:P\to\{d\in S_{\Q}(f)\;|\;d\equiv r\pmod* p\},\;(\alpha,\beta)\mapsto F(\alpha,\beta)\delta.\]
Given $(\alpha,\beta)\in P$, let $\lambda=\alpha/\beta$ and $d=F(\alpha,\beta)\delta$. Then $\beta^{\partial+\sigma} g(\lambda)=F(\alpha,\beta)$, so $\sqf(g(\lambda))=F(\alpha,\beta)$. Note that $F(\alpha,\beta)$ is relatively prime to $D$: if $\ell$ is a prime dividing $D$, then $\ell\nmid g(q)=F(q,1)$, and therefore $\ell\nmid F(\alpha,\beta)$ since $F(\alpha,\beta)\equiv F(q,1)\pmod\ell$. Thus $d$ is squarefree. Using \eqref{f_value} we obtain
\[\beta^{\partial+\sigma}f(a\lambda+b)=d(qu)^2p^{2v+2\epsilon},\] from which it follows that $\sqf(f(a\lambda+b))=d$, and therefore $d\in S_{\Q}(f)$. We claim that $d\equiv r\pmod p$. Since $\beta$ is a unit modulo $p$, $\lambda$ belongs to the local ring $\Z_{(p)}$. In this ring we have the congruence $a\lambda+b\equiv b\pmod{p^{2(v+\epsilon)+1}}$; hence, by the displayed equation above,
\[d(qu)^2p^{2v+2\epsilon}\equiv \beta^{\partial+\sigma}f(b)\pmod*{p^{2(v+\epsilon)+1}}.\]
The definition of $b$ implies that $f(b)\equiv hz_0^2p^{2(v+\epsilon)}\pmod{p^{2(v+\epsilon)+1}}$. It follows that $d(qu)^2\equiv\beta^{\partial+\sigma}hz_0^2\equiv\beta^{\partial+\sigma}rz_0^2\pmod p$. Since $qu\equiv z_0\pmod p$ and $\beta\equiv 1\pmod p$, we obtain $d\equiv r\pmod p$, as claimed. This proves that the map $\psi$ is well defined.

We end by showing that $\psi$ has finite fibers. For this purpose it suffices to show that $F$ can represent a given nonzero integer only finitely many times. If $F$ is irreducible, then, since $\deg F\ge\deg f\ge 3$, this follows from a well-known theorem of Thue. If $F$ is reducible, the proof of this finiteness statement is a straightforward exercise.
\end{proof}

\subsection*{Proof of Theorem \ref{unconditional_thm}}
Assuming that every irreducible factor of $f(x)$ has degree at most six, we must show that for every large enough prime $p$, and every integer $r$ not divisible by $p$, there exist infinitely many $d\in S_{\Q}(f)$ such that $d\equiv r\pmod p$.

By the results of \cite{krumm} mentioned in the introduction, we may assume that $\deg f\ge 3$. As seen in the proof of Theorem \ref{abc_thm}, if $p$ satisfies the conditions \eqref{p_sufficient_conditions}, then $S(r,0)$ holds for every integer $r\not\equiv 0\pmod p$. Applying Proposition \ref{unconditional_prop} we obtain the desired result.
\qed

\end{document}